\documentclass{article}

\usepackage{amsfonts}
\usepackage{amsmath}
\usepackage{theorem}
\usepackage{array}
\usepackage[a4paper]{geometry}
\usepackage{amssymb}
\usepackage{graphics}
\usepackage[usenames]{color}
\usepackage[all]{xy}
\usepackage{graphicx}
\usepackage{boxedminipage}

\newtheorem{theorem}{Theorem}[section]

\newtheorem{definition}[theorem]{Definition}
\newtheorem{remark}[theorem]{Remark}

\newtheorem{lemma}[theorem]{Lemma}

\newenvironment{proof}[1][Proof]{\textbf{#1.} }{\ \rule{0.5em}{0.5em}}

\newcommand{\n}{\mathfrak{n}}
\newcommand{\p}{\mathfrak{p}}
\newcommand{\F}{\mathbb{F}}

\DeclareMathOperator{\ord}{\,ord}

\begin{document}

\title{Good families of Drinfeld modular curves}
\author{Alp Bassa, Peter Beelen and Nhut Nguyen}
\date{}
\maketitle

\begin{abstract}
In this paper we investigate examples of good and optimal Drinfeld modular towers of function fields. Surprisingly, the optimality of these towers has not been investigated in full detail in the literature. We also give an algorithmic approach on how to obtain explicit defining equations for some of these towers and in particular give a new explicit example of an optimal tower over a quadratic finite field.
\end{abstract}

\section{Introduction} % use lowercase except for proper names
\label{intro}
Let $\F_q$ be a finite field with $q$ elements. For any absolutely irreducible, nonsingular (projective) algebraic curve $X$ defined over $\F_q$ the genus $g(X)$ and the number of rational points $N_1(X)$ satisfy the inequality $N_1(X) \le q+1 + 2\sqrt{q} g(X)$. This inequality is known as the Hasse--Weil bound. To investigate the asymptotic behaviour of such curves with increasing genus, Ihara introduced the quantity $$A(q):=\limsup_{g(X)\to \infty} \frac{N_1(X)}{g(X)},$$ where the limit is over all projective, absolutely irreducible, nonsingular algebraic curves defined over $\F_q$. It is known that $0< A(q) \le \sqrt{q}-1$, the first inequality being due to Serre \cite{serre}, while the second inequality is known as the Drinfeld--Vladut bound \cite{Vladut1983}. Combining the work of Ihara \cite{Ihara1982} and the Drinfeld--Vladut bound, one sees that $A(q)=\sqrt{q}-1$ if $q$ is a square. Note that for nonsquare values of $q$ the true value of $A(q)$ is currently unknown.

There exist a variety of constructions showing that $A(q)=\sqrt{q}-1$ if $q$ is a square. By the Drinfeld--Vladut bound it is sufficient to show that $A(q)\ge \sqrt{q}-1$ in this case. Ihara used families of Shimura modular curves for this purpose \cite{Ihara1982}, while Tsfasman--Vladut--Zink used families of (classical) modular curves (for $q=p^2$ and $q=p^4$) \cite{MANA:MANA19821090103}. Gekeler showed that also certain families of Drinfeld modular curves can be used \cite{gekeler2004a}. A different and completely explicit approach was presented by Garcia and Stichtenoth \cite{garcia1996a}. For any square $q$, they presented an explicitly defined family of curves $C_i$ (or rather towers of function fields $(F_i)_i$) defined over $\F_q$ for which the ratio $N_1(F_i)/g(F_i)$ tends to $\sqrt{q}-1$. Such families are called asymptotically optimal. This discovery led to an alternative approach to obtain lower bounds on $A(q)$ and by the explicit nature of their construction, the resulting function fields are more apt for applications in for example the theory of error-correcting codes \cite{Goppa1981170,MANA:MANA19821090103}. Despite the apparent difference of the constructions given in \cite{garcia1996a}, it was shown by Elkies that the same equations can be obtained using Drinfeld modular curves \cite{elkies2001a}. Conversely, the theory of modular curves can be used to produce explicitly defined families of curves \cite{Elkies98explicitmodular,bassa2014a}. The current work can be seen as a continuation and solidification of the work started in \cite{bassa2014a} to explicitly define families of Drinfeld modular curves. We will on occasion use the language of function fields rather than the more geometric language of curves to describe such families. %The article is organized as follows: First we give the necessary background in the preliminaries section. After that we will describe the main theoretical results, while we finish the article with a new explicit example of an optimal tower as well as with an algorithmic description on how to compute such Drinfeld modular towers.

\section{Preliminaries}

To put this work into the right context of Drinfeld modular curves, we briefly recall some notions that we will use in the remainder of the paper. See \cite{gekeler1986a} for a more detailed exposition on Drinfeld modular curves and \cite{Goss1996} for an exposition on Drinfeld modules. Let $F/\F_q$ be a function field with full constant field $\F_q$ and let $P$ be a place of degree $d$. Then we denote by $F_P$ the residue field of $P$. It is a finite field with $|F_P|:=q^d$ elements. For an integer $e \ge 1$, we denote by $F_P^{(e)}$ the algebraic extension of $F_P$ of degree $e$. In the theory of Drinfeld modules and Drinfeld modular curves one singles out a place $P_\infty$ of $F$ (playing the role of a place at ``infinity") and defines the ring $A$ as the ring of all functions in $F$ regular outside $P_\infty$. We will denote the degree of $P_\infty$ by $\delta$. Note that prime ideals of $A$ can be identified with places of $F$ distinct from $P_\infty$. For an ideal $\n \subset A$ we define $|\n|:=|A/\n|$ and $\deg \n:= \log_q |\n|$. In case $\n=(a)$ is a principal ideal, we write $\deg a:=\deg (a)$. In the special case of $F=\F_q(T)$ and $P_\infty$ the pole of $T$, one gets $\delta=1$ and $A=\F_q[T]$. In this case places of $F$ different from $P_\infty$ can be identified with monic irreducible polynomials and ideals of $A$ with monic polynomials.

Let $L$ be a field and $\iota: A \rightarrow L$ a homomorphism. The kernel of $\iota$ is called the \emph{$A$-characteristic} of $L$. Let $L\{\tau\}$ be the \emph{non-commutative polynomial ring} generated by the Frobenius endomorphism $\tau$ satisfying $\tau r = r^q \tau$ for all $r \in L$. Then an \emph{$A$-Drinfeld module} over $L$ of \emph{rank} $2$ is a homomorphism
\begin{align*}
\phi : A &\to L\{\tau\}\\
a &\mapsto \phi_a
\end{align*}
such that for all $a \in A\backslash\{0\}$, we have $\deg_\tau \phi_a = 2 \deg a$, and the constant term of $\phi_a$ is equal to $\iota(a)$. Elements of $L\{\tau\}$ can also be interpreted as linearized polynomials by replacing $\tau^i$ by $X^{q^i}$. This makes it possible to evaluate elements of $L\{\tau\}$ at elements of $\overline{L}$, the algebraic closure of $L$. Let $\n \subset A$ be an ideal of $A$, then we define $\phi[\n]$ to be the set of elements $x\in \overline{L}$ such that $\phi_a(x)=0$ for all $a \in \n$. This set is called the set of \emph{$\n$-torsion points} of the Drinfeld module $\phi$. If $\n$ is coprime with the $A$-characteristic of $L$, then $\phi[\n] \cong (A/\n)^{2}$ as an $A$-module. Two Drinfeld modules $\phi$ and $\psi$ with the same $A$-characteristic are called \emph{isogenous} if there exists $\lambda \in L\{\tau\}$ different from zero such that $\lambda\phi_a = \psi_a \lambda$ for all $a \in A$. The element $\lambda$ is called an \emph{isogeny}. The Drinfeld modules $\phi$ and $\psi$ are called \emph{isomorphic} if $\lambda$ can be chosen from $\overline{L} \backslash\{0\}$. If the kernel of the isogeny $\lambda$ is a free $A/\n$ module of rank one contained in $\phi[\n]$, then $\lambda$ is called an $\n$-isogeny.

For a non-zero monic polynomial $\n \in \F_q[T]$ Gekeler investigates in \cite{gekeler1979drinfeld} (among other things) the Drinfeld modular curve $Y_0(\n)$. The points on this curve parametrize isomorphism classes of pairs of $\F_q[T]$-Drinfeld modules of rank $2$ together with an $\n$-isogeny between them. Adding so-called cusps gives a projective algebraic curve $X_0(\n)$ defined over $F$ that in general however will not be absolutely irreducible. In case $\n=1$, the number of cusps is seen to be $(\delta \cdot h(F))^2$ while $X_0(1)$ has $\delta \cdot h(F)$ components \cite[VI.5]{gekeler1986a}. Here $h(F)$ denotes the class number of the function field $F$. This implies that the number of absolutely irreducible components of $X_0(\n)$ equals $\delta \cdot h(F)$. Equivalently, the number of components is equal to $h(A)$, the cardinality of the ideal class group of the ring $A$. By considering  the action of the ideal class group of $A$, one sees that the cusps are distributed equally among the absolutely irreducible components of $X_0(1)$, which implies that any such component contains exactly $\delta \cdot h(F)$ cusps. We will denote by $x_0(\n)$ an absolutely irreducible component of $X_0(\n)$. For any prime ideal of $A$ (corresponding to a place of $F$ different from $P_\infty$), one obtains by reduction an algebraic curve defined over a finite field. In case of  $A=\F_q[T]$ and $\delta=1$, the curve $X_0(\n)$ (as well as its reduction modulo any prime $P$ relatively prime to $\n$) is absolutely irreducible. By computing the precise formula for the genus and the number of rational points on reductions of $\F_q[T]$-Drinfeld modular curves $X_0(\n)$, Gekeler \cite{gekeler2004a} showed that for a series $(\n_k)_{k\in \mathbb{N}}$ of polynomials of $A$ coprime with an irreducible polynomial $P \in A$, and whose degrees tend to infinity, the family of Drinfeld modular curves $X_0(\n_k)/F_P$ attains the Drinfeld--Vladut bound when considered over $F_P^{(2)}$. In case $\n_k=T^k$ and $P=T-1$, explicit equations for the modular curves $X_0(T^k)$ were given in \cite{Elkies98explicitmodular}, while some more general examples (including defining equations in generic $A$-characteristic $0$) were given in \cite{bassa2014a}. For $A=\F_q[T]$ and $\delta=1$ the situation has therefore to a large extent been investigated both theoretically and explicitly. However, we will see that generalizations to other rings $A$ and values of $\delta$ are possible and that in some cases the resulting families of curves can be described by equations explicitly.

\section{Genus calculation of $x_0(\n)$}

In this section we will compute the genus of (an irreducible component of) the modular curve $X_0(\n)$. We put no restriction on the choice of function field $F$ and place $P_\infty$. A recipe for this genus computation is given in \cite{gekeler1986a} using results from \cite{gekeler1979drinfeld}. The recipe was carried out in \cite{gekeler1986a} in case $\n$ is a prime ideal. We will in this section carry out the computations for any ideal $\n$. The computations in \cite{gekeler1979drinfeld,gekeler1986a} are carried out over the field $C_\infty$, which is the completion of the algebraic closure of the completion of $F$ at $P_\infty$. For our purposes one therefore needs to check that the genus of $x_0(\n)$ does not change when changing the constant field. For $A=\F_q[T]$, this result is contained in \cite{schweizerconservative}. In our case, note that the only points that ramify in the cover $X(\n)/X(1)$ are the elliptic points of $X(1)$ and the cusps of $X(1)$. The residue field of a cusp is isomorphic to the Hilbert class field of $F$ \cite[Thm. 1.9 (ii), p.81]{gekeler1986a}, while the residue field of an elliptic point is a subfield of the Hilbert class field of $\mathbb{F}_{q^2}F$ \cite[Prop. 2.2, p.83]{gekeler1986a}. In either case, the residue field is a separable extension of the field $F$. Using Corollary 3.4.2 from \cite{goldschmidt}, we see that the argument given in \cite{schweizerconservative} carries over to our situation.

One of the ingredients in the genus expressions of $x_0(\n)$ involve the L-polynomial of the function field $F$, which we will denote by $P(t)$. Note that $P(1)=h(F)$, the class number of $F$. The following functions will also be useful:

\begin{definition}\label{def:usefulfunctions}
Let $\n \subset A$ be an ideal and suppose that $\n = \p_1^{r_1} \cdots \p_s^{r_s}$, for prime ideals $\p_1,\dots,\p_s$ and positive integers $r_1,\dots,r_s$. Writing $q_i:=|\p_i|$, we define
\begin{equation*}
\varphi(\n):=|(A/\n)^*| = \prod_{i=1}^{s}q_i^{r_i-1}(q_i-1), %= |\n|\prod_{i=1}^{s}(1-1/|\p_i|),
\end{equation*}
\begin{equation*}
\varepsilon(\n):=\prod_{i=1}^{s}q_i^{r_i-1}(q_i+1). % = |\n|\prod_{i=1}^{s}(1+ 1/|\p_i|),
\end{equation*}
and
\begin{equation*}
\kappa(\n):= \prod_{i=1}^{s}(q_i^{[r_i/2]} + q_i^{r_i - [r_i/2]-1}),
\end{equation*}
where $[r]$ denotes the integral part of a real number $r$.
\end{definition}
Using these notions, we will obtain that
\begin{theorem}\label{th:X0n}%\cite[3.4.18]{gekeler1979drinfeld}
Let $A$ and $\n$ be as above. In particular suppose that $\n = \p_1^{r_1} \cdots \p_s^{r_s}$, for prime ideals $\p_1,\dots,\p_s$ and positive integers $r_1,\dots,r_s$. Then we have
\begin{equation*}%\label{eq:X0M}
g(x_0(\n)) = 1 + \frac{(q^\delta-1)\varepsilon(\n)P(q)}{(q^2-1)(q-1)} -\frac{P(1)\delta(\kappa(\n) + 2^{s-1}(q-2))}{q-1} + \eta,
\end{equation*}
where $\eta = -P(-1)2^{s-1}q/(q+1)$ if $\delta$ is odd and all prime divisors of $\n$ are of even degree, $\eta=0$ otherwise.
\end{theorem}
Note that \cite[VII. 5.13]{gekeler1986a} (the case that $\n$ is a prime ideal) is a special case of this theorem.

The recipe outlined in \cite{gekeler1986a} consists of the following ingredients: first compute the genus of $x_0(1)$, then consider the cover $x_0(\n)/x_0(1)$. Since (like in the case of classical modular curves) this cover is not Galois in general, one studies a Galois cover $x(\n)/x_0(1)$ first. The curve $x(\n)$ is an irreducible component of the modular curve $X(\n)$, whose points correspond to isomorphism classes of $A$-Drinfeld modules $\phi$ of rank $2$ together with an isomorphism of $\phi[\n]$ with $(A/\n)^2$. Note that $X_0(1)=X(1)$ and that the points on this curve correspond to isomorphism classes of $A$-Drinfeld modules of rank $2$.

Since $x(\n)/x(1)$ is Galois, so is $x(\n)/x_0(\n)$. The Galois group of the cover $x(\n)/x(1)$, resp. $x(\n)/x_0(\n)$, is given by $G(\n)$, resp. $H(\n)$ defined as \cite[VII.5]{gekeler1986a}:
\begin{equation*} %\label{eq:Gn}
G(\n) %:= %\Gamma(1)/(\Gamma(\n)Z(\F_q)) \approxeq
:= \{\gamma \in \mathrm{GL}(2,A/\n) : \det \gamma \in \F_q^* \}/Z(\F_q)
\end{equation*}
and
\begin{equation*}
H(\n) %:=\Gamma_0(\n)/(\Gamma(\n)Z(\F_q))
:= \left\{ \left(\begin{matrix} %{cc}
a & b\\
0 & d
\end{matrix} \right) \in \mathrm{GL}(2,A/\n) : ad \in \F_q^* \right\}/Z(\F_q),
\end{equation*}
with
\[
Z(\F_q) := \left\lbrace \left(\begin{matrix} %{cc}
a & 0\\
0 & a
\end{matrix}\right) : a \in \F_q^*\right\rbrace.
\]
Before proceeding, we calculate the cardinalities of the groups $G(\n)$ and $H(\n)$. The latter cardinality is relatively easy, since in that case $a \in (A/\n)^*$ and $b \in A/\n$ can be chosen freely (leaving $q-1$ possibilities for $d$). Therefore, we have
\begin{equation}\label{eq:|H(n)|}
|H(\n)|= |(A/\n)^*|\cdot (q-1) \cdot |A/\n|/(q-1)  = \varphi(\n)|\n|.
\end{equation}
To count the cardinality of $G(\n)$, observe that
$$|\mathrm{SL}(2,A/\n)|=\frac{|\{\gamma \in \mathrm{GL}(2,A/\n) : \det \gamma \in \F_q^* \}|}{q-1},$$
since any nonzero value in $\F_q$ of the determinant is taken equally often when considering elements in $\{\gamma \in \mathrm{GL}(2,A/\n) : \det \gamma \in \F_q^* \}$. By definition of $G(\n)$, we obtain that
$$|G(\n)|=|\mathrm{SL}(2,A/\n)|.$$
The cardinality of $\mathrm{SL}(2,A/\n)$ is well known and can be computed using the Chinese remainder theorem. This approach gives that if $\n=\prod_i\p_i^{r_i}$ for prime ideals $\p_i \subset A$, then
$$|\mathrm{SL}(2,A/\n)|=\prod_i|\mathrm{SL}(2,A/\p_i^{r_i})|=\prod_i |\p_i|^{3r_i-2}(|\p_i|^2-1) =\varphi(\n)\varepsilon(\n)|\n|,$$ implying that
\begin{equation}\label{eq:|G(n)|}
|G(\n)|= \varphi(\n)\varepsilon(\n)|\n|.
\end{equation}

We now turn our attention again to the Galois cover $x(\n)/x(1)$. It was shown in \cite{gekeler1986a} that the only ramification in this cover occurs above the so-called elliptic points (with ramification index $q+1$) and the cusps of $x(1)$. Moreover, as mentioned before, the number of cusps on $x(1)$ equals $\delta h(F)$. The elliptic points were studied in \cite[V.4,VII.5]{gekeler1986a}: The number of elliptic points on $x(1)$ is $0$ if $\delta$ is even and $P(-1)$ if $\delta$ if odd, each with ramification index $q+1$ in the cover $x(\n)/x(1)$. We now write, just as before, $\n=\p_1^{r_1} \cdots \p_s^{r_s}$ for prime ideals $\p_1,\dots,\p_s$ of $A$ and positive integers $r_1,\dots,r_s$. Although $x(1)$ contains $P(-1)$ elliptic points if $\delta$ is odd, such an elliptic point does not give rise to ramification in the cover $x(\n)/x_0(\n)$ if any of the $\p_i$ has odd degree. If $\delta$ is odd and all prime ideals $\p_i$ occurring in the decomposition of $\n$ have even degree, %above each of the $P(-1)$ elliptic points of $x(1)$ lie exactly $2^s$ points of $x_0(\n)$ that branch in the cover
among all the points of $x_0(\n)$ that are lying above a given elliptic point of $x(1)$ there are exactly $2^s$ that are ramified in the covering $x(\n)/x_0(\n)$ (with ramification index $q+1$). This completely determines the behaviour of elliptic points as far as their role in the genus computation of $x(\n)$ and $x_0(\n)$ goes. To describe the behaviour of the cusps, we start by describing their ramification groups in $x(\n)/x(1)$ (following \cite[VII.5]{gekeler1986a}):

\begin{lemma}[Lemma 5.6 \cite{gekeler1986a}]
Let $$G(\n)_\infty:=\left\{\left(\begin{matrix} a & b \\ 0 & d \end{matrix}\right) \in \mathrm{GL}(2,A/\n) : a,d \in \F_q^* \right\}/Z(\F_q).$$
Then the stabilizers of all cusps of $x(\n)$ are conjugate in $G(\n)$ to $G(\n)_\infty.$
\end{lemma}
This means in particular that the ramification index in $x(\n)/x(1)$ of any cusp equals $|G(\n)_\infty|=(q-1)^2|\n|/(q-1)=(q-1)|\n|.$ The cardinality of the first, resp. second, ramification group of any cusp is then calculated in \cite[Lemma 5.7]{gekeler1986a} to be $|\n|$, resp. $1$. This means that the different exponent for a cusp equals $(q-1)|\n|-1+|\n|-1=q|\n|-2$. Combining this information concerning the ramification groups of the cusps with the description of the ramification behaviour of the elliptic points, makes the computation of the genus of $x(\n)$ completely feasible using the Riemann--Hurwitz genus formula. The result (given in slightly less explicit form in \cite[Theorem 5.11]{gekeler1986a}) is:
\begin{equation}\label{eq:Xn}
g(x(\n)) =1 + \frac{(q^\delta-1)P(q)}{(q^2-1)(q-1)}\varphi(\n)\varepsilon(\n)|\n| -  \frac{\delta P(1)}{q-1}\varphi(\n)\varepsilon(\n).
\end{equation}

The ramification behaviour of the cusps is more complicated in the cover $x(\n)/x_0(\n)$. However, in \cite[VII.5]{gekeler1986a} (with reference to \cite[3.4.15]{gekeler1979drinfeld}) the total contribution to the Riemann--Hurwitz genus formula for the cover $x(\n)/x_0(\n)$ of all cusps of $x(\n)$ lying above a single cusp of $x(1)$ is computed to be
\begin{equation}\label{eq:cuspcontr}
(q-1)^{-1}\varphi(\n)(2|\n| \kappa(\n) + 2^s(q-2) |\n| - 2\varepsilon(\n)).
\end{equation}
We now have all the ingredients needed for the proof of Theorem \ref{th:X0n}
\begin{proof}
For any point $P$ of $x(\n)$, we denote by $e(P)$, resp. $d(P)$, the ramification index, resp. different exponent, in the cover $x(\n)/x_0(\n)$. Since the only ramified points in the cover $x(\n)/x(1)$ are the cusps and the elliptic points (if these exist), applying the Riemann--Hurwitz genus formula for the cover $x(\n)/x_0(\n)$ we obtain:
\begin{equation}\label{eq:RH}
2g(x(\n))-2=\varphi(\n)|\n|(2g(x_0(\n))-2)+\sum_{P \ \makebox{cusp}} d(P)+\sum_{\substack{P \ \makebox{elliptic} \\ \makebox{point}}}d(P).
\end{equation}
The sum concerning the elliptic points is zero if no such points exist and therefore:
$$\sum_{\substack{P \ \makebox{elliptic} \\ \makebox{point}}} d(P)=0,$$
if $\delta$ is even or if there exists $\p_i$ of odd degree. Otherwise, as we have described previously, above each of the $P(-1)$ cusps of $x(1)$ lie exactly $2^s$ points of $x_0(\n)$ that ramify with ramification index $q+1$ in $x(\n)/x_0(\n)$. This implies that
$$\sum_{\substack{P \ \makebox{elliptic} \\ \makebox{point}}} d(P)=\sum_{\substack{P \ \makebox{elliptic} \\ \makebox{point}}} q=P(-1)2^s q |\n|\varphi(\n)/(q+1),$$
if $\delta$ is odd and all prime divisors of $\n$ have even degree.

The summation over the points lying over any of the $\delta h(F)$ cusps of $x(1)$ can be dealt with using Equation \eqref{eq:cuspcontr}. We obtain that
$$\sum_{P \ \makebox{cusp}} d(P) = \delta h(F)(q-1)^{-1}\varphi(\n)(2|\n| \kappa(\n) + 2^s(q-2) |\n| - 2\varepsilon(\n)).$$
Substituting these values in Equation \eqref{eq:RH} and using Equation \eqref{eq:Xn}, Theorem \ref{th:X0n} follows.
\end{proof}

\section{Rational points on reductions of Drinfeld modular curves}

In this section, we combine the previously described genus computation of the curves $x_0(\n)$ with the fact that reductions of these curves have many rational points (when the field of definition is chosen properly). We will show that for any sequence of ideals $(\n_k)_{k\ge 1}$ such that $\deg \n_k \to \infty$ as $k \to \infty$, the corresponding family of reductions of Drinfeld modular curves $(x_0(\n_k))_k$ has good asymptotic properties. In \cite{taelman2006} the (reductions of the) curves $x_0(\n)$ were also investigated in case $\n$ is a principal ideal, using a different method inspired by \cite{Ihara1982}. Our approach is to use, for any ideal $\n$, results from \cite{Gekeler1990a} to estimate the number of rational points on the reduction of $x_0(\n)$ and to use the explicit genus formula for $g(x_0(\n))$ from the previous section.

While the curves $X_0(\n)$ themselves are defined over the function field $F$ (and a component $x_0(\n)$ over an extension field of $F$), a model can be found that can be reduced modulo prime ideals of the ring $A$. This reduction is known to be good if $P \subset A$ is a prime ideal which is coprime with the ideal $\n$. Thus, reduction modulo $P$ gives rise to a curve (as before not necessarily absolutely irreducible) that is defined over the finite field $A/P$. For convenience we write $F_P:=A/P$ and denote by $F_P^{(m)}$ the degree $m$ extension of $F_P$. In case $A=\F_q[T]$, these reduced Drinfeld modular curves have many rational points over $F_P^{(2)}$ (essentially corresponding to supersingular $A$-Drinfeld modules), but it turns out that in general the situation is slightly more complicated. As a matter of fact the supersingular Drinfeld modules in $A$-characteristic $P$ are in general defined over the field $F_P^{(2e)}=\F_{q^{2de}}$ with $d=\deg P$ and $e=\mathrm{ord }\, P$, the order of the ideal $P$ in the ideal class group of the ring $A$ \cite[Section 4]{Gekeler1990a}.

More precisely, in \cite{Gekeler1990a} it was shown that for a prime ideal $P \subset A$ with $d:=\deg P$, the number $N(P)$ of isomorphism classes of supersingular $A$-Drinfeld modules in $A$-characteristic $P$ equals $N(P)=h_1(P)+h_2(P)$ with
\begin{equation}\label{eq:numberh1}
h_1(P):=\left\{
\begin{array}{rl}
\delta P(1)\left( P(q)\frac{(q^\delta-1)(q^d-1)}{(q^2-1)(q-1)}-\frac{P(-1)}{q+1} \right),& \makebox{if $d$ and $\delta$ are odd,}\\
\\
\delta P(1)P(q)\frac{(q^\delta-1)(q^d-1)}{(q^2-1)(q-1)} & \makebox{otherwise,}
\end{array}
\right.
\end{equation}
and
\begin{equation}\label{eq:numberh2}
h_2(P):=\left\{
\begin{array}{rl}
\delta P(1)P(-1), & \makebox{if $d$ and $\delta$ are odd,}\\
\\
0 & \makebox{otherwise.}
\end{array}
\right.
\end{equation}
Each isomorphism class of a supersingular $A$-Drinfeld module gives rise to a rational point (which we will call a supersingular point) on the curve $X(1)$, if the field of definition is taken to be $F_P^{(2e)}$. Using the action given by the class group of $A$ on the absolutely irreducible components of $X(1)$, one  sees that the supersingular points are equidistributed among all $\delta P(1)$ components of $X(1)$. These observations enable us to give a lower bound on the number of rational points on $x_0(\n)$:
\begin{theorem}
Let $\n \subset A$ be an ideal prime to the $A$-characteristic $P$ and suppose that $\n = \p_1^{r_1} \cdots \p_s^{r_s}$, for prime ideals $\p_1,\dots,\p_s$ and positive integers $r_1,\dots,r_s$. Moreover, denote by $d:=\deg P$ and $e:=\mathrm{ord } \, P$. Consider over the finite field $F_P^{(2e)}$ a component $x_0(\n)$ of $X_0(\n)$ and denote by $N_1(x_0(\n))$ its number of rational points. Then if $d,\delta$ are odd, and $\deg \p_i$ is even for all $i$, we have
$$N_1(x_0(\n)) \ge \varepsilon(\n)P(q)\frac{(q^\delta-1)(q^d-1)}{(q^2-1)(q-1)}+P(-1)2^s \frac{q}{q+1},$$
while otherwise
$$N_1(x_0(\n)) \ge \varepsilon(\n)P(q)\frac{(q^\delta-1)(q^d-1)}{(q^2-1)(q-1)}.$$
\end{theorem}
\begin{proof}
All points of $x_0(\n)$ lying above one of the $N(P)/(\delta P(1))$ supersingular points of $x(1)$ are rational, but not necessarily unramified in the covering $x_0(\n)/x(1)$. The reason for this is that the elliptic points are supersingular points if (and only if) both $\delta$ and $d$ are odd \cite[Lemma 7.2]{Gekeler1990a}. However, any elliptic point has ramification index either one, or $q+1$ in the cover $x_0(\n)/x(1)$. Moreover, from \cite[V.4,VII.5]{gekeler1986a} we see that if $\delta$ is odd and all prime ideals $\p_i$ occurring in the decomposition of $\n$ have even degree, %above each of the $P(-1)$ elliptic points of $x(1)$ there lie exactly $2^s$ points of $x_0(\n)$ that branch in the cover
among all the points of $x_0(\n)$ that are lying above a given elliptic point of $x(1)$ there are exactly $2^s$ that are ramified in the covering $x(\n)/x_0(\n)$ (with ramification index $q+1$). The latter statement is equivalent to saying that these $2^s$ points of $x_0(\n)$ have ramification index $1$ in $x_0(\n)/x(1)$. Counting the number of points of $x_0(\n)$ lying above the supersingular points of $x(1)$ now is direct and yields the stated lower bound on $N_1(x_0(\n))$.
\end{proof}

From Theorem \eqref{th:X0n} we get the following asymptotic result:
\begin{theorem}\label{th:lim}
Let $A$ be any ring of functions regular outside a fixed place $\infty$ of degree $\delta$. Let $P \subset A$ be a prime ideal of degree $d$ and order $e$ and further let  $(\n_k)_{k \ge 1}$ be a series of ideals relatively prime to $P$. The family of reductions of Drinfeld modular curves $(x_0(\n_k))_k$ when defined over $\F_{q^{2de}}$ satisfies $$\lim_{k \to \infty} \frac{N_1(x_0(\n_k))}{g(x_0(\n_k))} \ge q^d-1.$$
\end{theorem}

\begin{remark}
The lower bound given in Theorem \ref{th:lim} is sharp in case $P$ is a principal ideal, since in this case $e=1$ and the given lower bound is equal to the Drinfeld--Vladut upper bound. If $A=\F_q[T]$ (in particular $\delta=1$), the ideal class group of $A$ is trivial, implying that any family of reductions of Drinfeld modular curves as in Theorem \ref{th:lim} has optimal asymptotic properties. This particular case was shown in \cite{gekeler2004a}. If $P$ is not principal, the resulting families will be asymptotically good, but not optimal. Note that in \cite{taelman2006} this subtlety is missing.
\end{remark}

\section{A recursive description of a Drinfeld modular tower}\label{sec:recursive}

In this section we will illustrate Theorem \ref{th:lim} by describing some families of Drinfeld modular curves $(x_0(\n_k))_k$ more explicitly. In case $\n_k=\p^k$ for a fixed prime ideal $\p$ of $A$, this can be done in a recursive way (in fact $\p$ could be any non-trivial ideal, but we will assume primality for simplicity). The reason for this is similar to the reasoning presented in \cite{Elkies98explicitmodular,elkies2001a}, but is somewhat more involved due to the fact that the curves $X(1)$ and $X_0(\p^k)$ are not absolutely irreducible in general. Therefore, we go through the argument in the following.

A point on $X_0(\p)$ corresponds to an isomorphism class $[\phi,\psi]$ of a pair of $\p$-isogenous $A$-Drinfeld modules of rank two. Therefore, there are two possible maps, say $\pi_1$ and $\pi_2$, from $X_0(\p)$ to $X(1)$, see Figure \ref{fig:correspondence}, since one can send $[\phi,\psi]$ to $[\phi]$ or $[\psi]$ (the isomorphism class of $\phi$ or that of $\psi$). Since a $\p$-isogeny corresponds to a cyclic submodule of the $\p$-torsion points of $\phi$, the degree of the first map is $|\p|+1$. By symmetry, the degree of the second map is also $|\p|+1$.

The image of a fixed absolutely irreducible component $x_0(\p)$ of $X_0(\p)$ under either $\pi_1$ or $\pi_2$, will be an absolutely irreducible component of $X(1)$, but not necessarily the same one. We denote these components by $x^{1}(1)$ and $x^{2}(1)$. We can then view $x_0(\p)$ as a curve lying inside $x^{1}(1) \times x^{2}(1)$. Once an explicit description of the components of $x^{1}(1)$ and $x^{2}(1)$ is available, the map $\pi_1 \times \pi_2: x_0(\p) \to x^{1}(1) \times x^{2}(1)$ defined by $[\phi,\psi] \mapsto ( [\phi],[\psi] )$, can be in principle be used to describe the curve $x_0(\p)$ explicitly by equations. However, in practice it is very convenient to assume that the genera of the components of $X(1)$ are zero. In this case, a component $x^{i}(1)$ can just be described using a single variable $u_i$, which one can think of as a $j$-invariant of an $A$-Drinfeld module. In this case a component of $X_0(\p)$ can be described using a bivariate polynomial $\Phi(u_1,u_2)$ of bi-degree $(|\p|+1,|\p|+1)$ (that is, of degree $|\p|+1$ in either of the two variables $u_1$ and $u_2$). Note that for $\n=1$, Equation \eqref{eq:Xn} states that
\begin{equation}\label{eq:X1}
g(x(1)) = 1 + (q^2-1)^{-1}\left( \frac{q^\delta-1}{q-1}P(q) - \frac{q(q+1)}{2}\delta P(1) + \eta \right),
\end{equation}
where $\eta = -q(q-1)P(-1)/2$ for $\delta$ odd, $\eta=0$ otherwise. As a matter of fact, this formula was stated in \cite[VI.5.8]{gekeler1986a} and was used as a key ingredient there to showing Equation \eqref{eq:Xn}. Using Equation \eqref{eq:X1}, one readily sees that $g(x(1))=0$ if $F=\F_q(T)$ and $\delta \in \{1,2,3\}$ or if $F$ is the function field of an elliptic curve and $\delta=1$. For simplicity, we assume from now on that we are in one of these situations, though the general considerations below remain valid in the general case as well. However, finding explicit equations is only possible if (the function field of) the curve $x(1)$ can be given explicitly, which is trivial if it has genus zero.

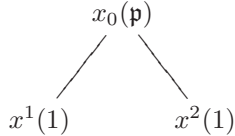
\begin{figure}[h]
\begin{displaymath}
\xymatrix@C=2pt{
%& & x^1_0(\p^2)\ar@{-}[dl]\ar@{-}[dr] & & \ar@{-}[dl]& &\cdots & & \ar@{-}[dr]& & x^1_0(\p^2)\ar@{-}[dl]\ar@{-}[dr] & & & & \\
& x_0(\p)\ar@{-}[dl]\ar@{-}[dr]\\% & & x^2_0(\p)\ar@{-}[dl]\ar@{-}[dr] & &  & \cdots &  & & x^1_0(\p)\ar@{-}[dl]\ar@{-}[dr] & & x^2_0(\p)\ar@{-}[dl]\ar@{-}[dr] & &  &\\
  x^{1}(1) & & x^2(1) %& & x^3(1) & & \cdots & & x^1(1) & & x^2(1) & & x^3(1) & & %\cdots
}
\end{displaymath}
\caption{A correspondence of modular curves.}
\label{fig:correspondence}
\end{figure}

A description of $x_0(\p^2)$ (a component of $X_0(\p^2)$) can now be obtained relatively easily. A point on $X_0(\p^2)$ corresponds to an isomorphism class $[\phi_1,\phi_3]$ of a pair of $\p^2$-isogenous $A$-Drinfeld modules of rank two. Let $\mu: \phi_1 \to \phi_3$ be the corresponding $\p^2$-isogeny. Then there exists a $A$-Drinfeld module $\phi_2$ of rank two and $\p$-isogenies $\lambda_1: \phi_1 \to \phi_2$ and $\lambda_2: \phi_2 \to \phi_3$ such that $\mu=\lambda_2 \circ \lambda_1$. The isomorphism class of $[\phi_i]$ will correspond to a point on a component $x^i(1)$ of $X(1)$. This means that we can map $x_0(\p^2)$ to $x^1(1) \times x^2(1) \times x^3(1)$. Note that both $[\phi_1,\phi_2]$ and $[\phi_2,\phi_3]$ correspond to points on $X_0(\p)$, lying on certain components, say $x_0^1(\p)$ and $x^2_0(\p)$. Using the above procedure, we can describe these two components as the zero set of polynomials $\Phi^1(u_1,u_2)$ and $\Phi^2(u_2,u_3)$, both of bi-degree $(|\p|+1,|\p|+1)$. This means that image of the map from $x_0(\p^2)$ to $x^1(1) \times x^2(1) \times x^3(1)$ is part of the zero set of the polynomials $\Phi^1(u_1,u_2)$ and $\Phi^2(u_2,u_3)$. However, this zero set turns out to be too large. The reason for this is that if $(\phi_1,\phi_2)$ and $(\phi_2,\phi_3)$ are two pairs of $\p$-isogenous $A$-Drinfeld modules of rank two, with $\p$-isogenies denoted by $\lambda_1$ and $\lambda_2$, then $\lambda_2 \circ \lambda_1$ is either a $\p^2$-isogeny, or has kernel isomorphic to $A/\p \times A/\p$. Here we used that $\p$ is a prime ideal. The latter case gives rise to additional elements in the zero set of $\Phi^1(u_1,u_2)$ and $\Phi^2(u_2,u_3)$. However, this issue is rather easy to resolve: We work over the function field of $x_0^1(\p)$, which we can construct using the polynomial $\Phi^1(u_1,u_2)$. The polynomial $\Phi^2(u_2,u_3)$, viewed as a univariate polynomial in $u_3$ and coefficients in the function field of $x_0^1(\p)$, %apparently is not absolutely irreducible. Since it
has degree $|\p|+1$ in $u_3$ while the extension degree of $X_0(\p^2) / X(1)$ is $\varepsilon(\p^2)=(|\p|+1)|\p|$. Then the polynomial $\Phi^2(u_2,u_3)$ is not absolutely irreducible and has a (for degree reasons necessarily unique) component of degree $|\p|$ in $u_3$. This component can then be used to construct (the function field of) $x_0(\p^2)$, also see Figure \ref{fig:level2}.

\begin{figure}[h]
\begin{displaymath}
\xymatrix@C=2pt{
& & x_0(\p^2)\ar@{-}[dl]\ar@{-}[dr] \\% & & \ar@{-}[dl]& &\cdots & & \ar@{-}[dr]& & x^1_0(\p^2)\ar@{-}[dl]\ar@{-}[dr] & & & & \\
& x^1_0(\p)\ar@{-}[dl]\ar@{-}[dr] & & x^2_0(\p)\ar@{-}[dl]\ar@{-}[dr] \\% & &  & \cdots &  & & x^1_0(\p)\ar@{-}[dl]\ar@{-}[dr] & & x^2_0(\p)\ar@{-}[dl]\ar@{-}[dr] & &  &\\
  x^{1}(1) & & x^2(1) & & x^3(1) % & & \cdots & & x^1(1) & & x^2(1) & & x^3(1) & & %\cdots
}
\end{displaymath}
\caption{Recursive description of $x_0(\p^2)$.}
\label{fig:level2}
\end{figure}
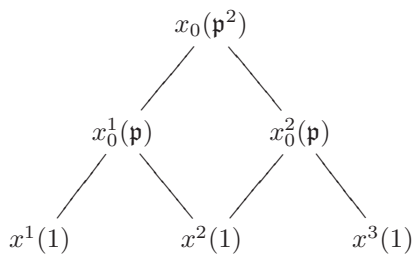

Iterating this procedure gives rise to an explicit recursive description of $x_0(\p^k)$ for any $k \ge 1$. One effectively just increases the size of the pyramid in Figures \ref{fig:correspondence} and \ref{fig:level2}. Note that since $X(1)$ only has finitely many absolutely irreducible components, ultimately the same components will start to occur, see Figure \ref{fig:tower}.

\begin{figure}[h]
\begin{displaymath}
\xymatrix@C=2pt{
& & & \ar@{-}[dl] & & &\cdots & & & \ar@{-}[dr] & &  & & &\\
& & x^1_0(\p^2)\ar@{-}[dl]\ar@{-}[dr] & & \ar@{-}[dl]& &\cdots & & \ar@{-}[dr]& & x^1_0(\p^2)\ar@{-}[dl]\ar@{-}[dr] & & & & \\
& x^1_0(\p)\ar@{-}[dl]\ar@{-}[dr] & & x^2_0(\p)\ar@{-}[dl]\ar@{-}[dr] & &  & \cdots &  & & x^1_0(\p)\ar@{-}[dl]\ar@{-}[dr] & & x^2_0(\p)\ar@{-}[dl]\ar@{-}[dr] & &  &\\
x^{1}(1) & & x^2(1) & & x^3(1) & & \cdots & & x^1(1) & & x^2(1) & & x^3(1) & & %\cdots
}
\end{displaymath}
\caption{The pyramid of Drinfeld modular curves.}
\label{fig:tower}
\end{figure}
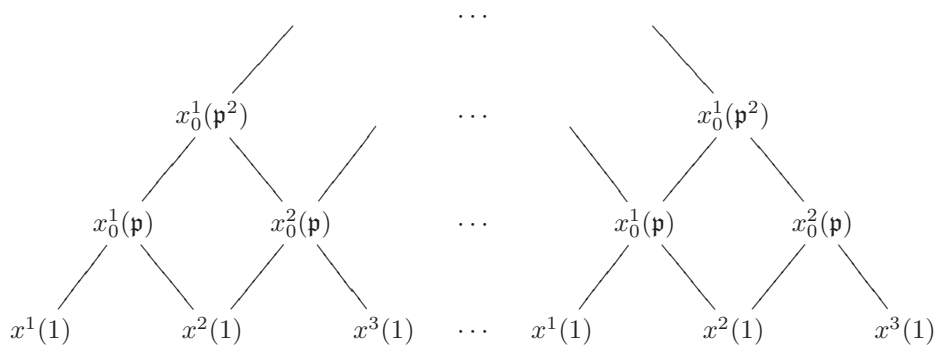

In case $A=\F_q[T]$, $\delta=1$, $\p=T$ and $A$-characteristic $T-1$, explicit equations were found in \cite{Elkies98explicitmodular}. In this case all curves $X(1),X_0(\p^k)$ are absolutely irreducible, so there is no need to keep track of components or to distinguish between $X_0(T^k)$ and one of its components $x_0(T^k)$. The curve $X_0(T)$ can be described using the Drinfeld modular polynomial $\Phi_T(u_1,u_2)$. However, the approach in \cite{Elkies98explicitmodular} exploits the fact that the genera of the curves $X_0(T)$ and $X_0(T^2)$ are zero. Compared to our approach this means that the "pyramid" in Figure \ref{fig:tower} starts at $X_0(\p^2)$, but otherwise the recursive description is similar: The points on the curve $X_0(T^k)$ are identified with points in $X_0(T^2) \times \cdots \times X_0(T^2)$, while each of the component curves $X_0(T^2)$ can be described using a single parameter $v_i$. For more details see \cite{Elkies98explicitmodular,bassa2014a}.

\section{An new explicit example of an optimal Drinfeld modular tower}

In \cite{bassa2014a} some examples of good towers were found following the above approach, including one where the function field $F$ was the function field of an elliptic curve and $\delta=1$. More precisely, in the latter example in \cite{bassa2014a} one had $F=\F_2(X,Y)$ with $X$ transcendental over $\F_2$ and $Y^2+Y=X^3+X$, while ``infinity'' was chosen to be the place at infinity of this elliptic curve, implying that $\delta=1$. The ring $A$ is then easily seen to be $\F_2[X,Y] \cong \F_2[T,S]/\langle S^2+S+T^3+T \rangle$. A description was given of the tower $X_0(\p^k)$ with $\p:=\langle X+1,Y+1\rangle \subset A$ and $A$-characteristic $P:=\langle X,Y \rangle$. Note that $\deg P=1$, since $P$ is a rational point on the elliptic curve, and $\ord P = 5$, since the elliptic curve has $5$ rational points, meaning that the group of rational points is cyclic of order $5$. It was shown in \cite{bassa2014a} by explicit computation that the tower $X_0(\p^k)$ (in $A$-characteristic $\langle T,S \rangle$) has limit at least $1$ when the constant field is set to $\F_{2^{10}}$. This result is confirmed by Theorem \ref{th:X0n}. In this section we will in a similar way as in \cite{bassa2014a} describe an explicit example of an optimal tower. Contrary to the example referred to above and motivated by Theorem \ref{th:X0n}, the choice of $A$-characteristic $P$ is now made such that $\ord P=1$, implying that the resulting tower is optimal. The point with this example is not to give another optimal tower, but to show an explicit description is within reach. Such a description is useful for applications in for example coding theory.

More precisely, we will consider the following setting:

\begin{enumerate}
\item $F/\mathbb{F}_q := \F_2(X,Y)/\F_2$, where $Y^2+XY+X^2=X$ and $X$ is transcendental over $\F_2$,
\item $A := \F_2[X,Y]$, implying $\delta=2$.
\item the $A$-characteristic $P$ is the principal prime ideal $\langle X^2+X+1 \rangle \subset A$.
\end{enumerate}

Note that the function field $F$ has genus $0$, implying that the L-polynomial $P(t)$ occurring in the zeta function of $F$ is simply $P(t)=1$. Therefore the curve $X(1)$ has $\delta P(1)=2$ absolutely irreducible components, say $x^1(1)$ and $x^2(1)$ both of genus $0$ according to Equation \ref{eq:X1}. Since for the given choice of $P$ we have $\ord P =1$ (since $P$ is a principal ideal) and $\deg P=4$, Theorem \ref{th:X0n} implies that, for any choice of prime ideal $\p \subset A$ coprime with the $A$-characteristic $P$, the limit of the resulting family of curves $(X_0(\p^k))_{k}$ when defined over the finite field $\F_{2^8}$ equals $\sqrt{2^8}-1=15$. In other words, the resulting family of curves is optimal over $\F_{2^8}$.

We start by indicating how to describe $A$-Drinfeld modules explicitly. An $A$-Drinfeld module of rank $2$ is symbolically determined by
\begin{align*}
\phi_X &= g_0\tau^4 + g_1\tau^3 + g_2\tau^2 + g_3\tau + \iota(X),\\
\phi_Y &= h_0\tau^4 + h_1\tau^3 + h_2\tau^2 + h_3\tau + \iota(Y).
\end{align*}

Since we have chosen the principal prime ideal $\langle X^2+X+1 \rangle$ as $A$-characteristic, we have $\iota(X)^2+\iota(X)+1=0$ and, using the equation of the curve, $\iota(Y)^2+\iota(X)\iota(Y)+\iota(X)^2=\iota(X)$. For convenience we will write $$x:=\iota(X) \ \makebox{ and } \ y:=\iota(Y).$$ We see that $x=\iota(X) \in \F_4$ and $y=\iota(Y) \in \F_{16}$. The remaining coefficients also satisfy several algebraic relations, stemming from the fact that $\phi_X \phi_Y = \phi_Y \phi_X$ and $\phi_{Y^2+XY+X^2-X}=0$. Indeed, any choice of $g_0,\dots,h_3$ satisfying these relations gives rise to a Drinfeld module. The equation $\phi_X \phi_Y = \phi_Y \phi_X$, implies that:
\begin{align}
g_0h_0^{q^4} &= h_0g_0^{q^4} \label{eq:normalize}\\
g_0h_1^{q^4} + g_1h_0^{q^3} &= h_0g_1^{q^4} + h_1g_0^{q^3} \label{eq:add4}\\
g_0h_2^{q^4} + g_1h_1^{q^3} + g_2h_0^{q^2} &= h_0g_2^{q^4} + h_1g_1^{q^3} + h_2g_0^{q^2} \label{eq:add3}\\
g_0h_3^{q^4} + g_1h_2^{q^3} + g_2h_1^{q^2} + g_3h_0^q &= h_0g_3^{q^4} + h_1g_2^{q^3} + h_2g_1^{q^2}+h_3g_0^q \label{eq:add2}\\
g_1h_3^{q^3} + g_2h_2^{q^2} + g_3h_1^q + x h_0 &= h_0x^{q^4} + h_1g_3^{q^3} + h_2g_2^{q^2} + h_3g_1^q \label{eq:add1}\\
g_1y^{q^3} + g_2h_3^{q^2} + g_3h_2^q + xh_1 &= h_1x^{q^3} + h_2g_3^{q^2} + h_3g_2^q + yg_1 \label{eq:g1}\\
g_2y^{q^2} + g_3h_3^q + xh_2 &= h_2x^{q^2} + h_3g_3^q + yg_2 \label{eq:g2} \\
g_3y^q+xh_3 &= h_3x^q + yg_3 \label{eq:g3}
\end{align}

Note that in this section we assume that $q=2$. Similarly the equation $\phi_{Y^2+XY+X^2-X}=0$ gives rise to algebraic relations. From Equations \eqref{eq:g3}, \eqref{eq:g2} and \eqref{eq:g1}, one sees that the three variables $g_3,g_2,g_1$ can be expressed in the three variables $h_3,h_2,h_1$. After eliminating $g_1,g_2,g_3$ in this way, Equations \eqref{eq:add1}, \eqref{eq:add2}, \eqref{eq:add3}, \eqref{eq:add4} give rise to pairs of polynomials in $h_1$. These polynomials turn out to have a very special form: they are linearized polynomials in $h_1$ plus a constant term. Therefore, we can use the $q$-linearized variant of the Euclidean algorithm to eliminate the variable $h_1$ very efficiently, thus avoiding a lengthy Groebner basis computation. Finally we may use Equation \eqref{eq:normalize} to normalize the leading coefficients $g_0$ and $h_0$ by putting $h_0=1$ and choosing $g_0\in \F_4$ such that $g_0^2+g_0+1=0$. %There are two possible choices for $g_0$, which turns out to reflect the fact that curve $X(1)$ has two components.
We are then left with an explicit algebraic equation relating $h_2$ and $h_3$, say $f(h_2,h_3)=0$, with coefficients in $\F_{16}$. The equation is a bit lengthy, but we state it for the sake of completeness:
\begin{equation*}
\footnotesize
\begin{split}
f(h_2,h_3)& = h_2^{30} + (xy + x)h_2^{29}h_3^{3} + (y + x)h_2^{27}h_3^{9} + (xy + 1)h_2^{26}h_3^{12} + (y + 1)h_2^{25} + \\
&\hspace{-0.8cm} (xy + x)h_2^{24}h_3^{18} + (x^{2}y + x^{2})h_2^{24}h_3^{3} +
        yh_2^{23}h_3^{21} + (x^{2}y + 1)h_2^{23}h_3^{6} + x^{2}yh_2^{22}h_3^{9} + \\
&\hspace{-0.8cm} (xy + 1)h_2^{21}h_3^{27} + (x^{2}y + x)h_2^{21}h_3^{12} + h_2^{20}h_3^{30} + (y + 1)h_2^{20}h_3^{15} + (xy + 1)h_2^{20} + \\
&\hspace{-0.8cm} (x^{2}y + x^{2})h_2^{19}h_3^{18} + yh_2^{18}h_3^{36} + (xy + x)h_2^{18}h_3^{6} + (y + x)h_2^{17}h_3^{39} + (y + x^{2})h_2^{17}h_3^{24} + \\
&\hspace{-0.8cm} xh_2^{17}h_3^{9} + (x^{2}y + 1)h_2^{16}h_3^{27} + xyh_2^{16}h_3^{12} + h_2^{15}h_3^{45} + (y + 1)h_2^{15}h_3^{30} + xyh_2^{15}h_3^{15} + \\
&\hspace{-0.8cm} (y + x)h_2^{15} + (x^{2}y + x^{2})h_2^{14}h_3^{33} + (y + 1)h_2^{14}h_3^{18} + h_2^{14}h_3^{3} + yh_2^{13}h_3^{51} + xyh_2^{13}h_3^{36} + \\
&\hspace{-0.8cm} xh_2^{13}h_3^{21} + (xy + x)h_2^{13}h_3^{6} + (y + x)h_2^{12}h_3^{54} + x^{2}yh_2^{12}h_3^{39} + (x^{2}y + x)h_2^{12}h_3^{9} + \\
&\hspace{-0.8cm} (x^{2}y + x)h_2^{11}h_3^{42} + (y + x^{2})h_2^{11}h_3^{27} + xh_2^{11}h_3^{12} + h_2^{10}h_3^{60} + (y + x^{2})h_2^{10}h_3^{45} + xh_2^{10}h_3^{30} + \\
&\hspace{-0.8cm} (y + x)h_2^{10}h_3^{15} + (xy + 1)h_2^{10} + (xy + x)h_2^{9}h_3^{63} + x^{2}yh_2^{9}h_3^{48} + (xy + x)h_2^{9}h_3^{33} + \\
&\hspace{-0.8cm} (xy + 1)h_2^{9}h_3^{18} + (xy + x)h_2^{9}h_3^{3} + xyh_2^{8}h_3^{51} + (x^{2}y + x)h_2^{8}h_3^{36} + (xy + x)h_2^{8}h_3^{21} + \\
&\hspace{-0.8cm} (y + x)h_2^{8}h_3^{6} + (y + x)h_2^{7}h_3^{69} + (y + x^{2})h_2^{7}h_3^{54} + (x^{2}y + 1)h_2^{7}h_3^{39} + (xy + 1)h_2^{7}h_3^{24} + \\
&\hspace{-0.8cm} xh_2^{7}h_3^{9} + (xy + 1)h_2^{6}h_3^{72} + xyh_2^{6}h_3^{42} + (xy + 1)h_2^{6}h_3^{27} + (y + x^{2})h_2^{6}h_3^{12} + xh_2^{5}h_3^{60} + \\
&\hspace{-0.8cm} (xy + 1)h_2^{5}h_3^{45} + h_2^{5}h_3^{30} + (xy + x^{2})h_2^{5}h_3^{15} + (y + 1)h_2^{5} + (xy + x)h_2^{4}h_3^{78} + yh_2^{4}h_3^{48} + \\
&\hspace{-0.8cm} (x^{2}y + x)h_2^{4}h_3^{33} + (xy + x)h_2^{4}h_3^{18} + x^{2}h_2^{4}h_3^{3} + yh_2^{3}h_3^{81} + xyh_2^{3}h_3^{66} + xh_2^{3}h_3^{51} + \\
&\hspace{-0.8cm} (x^{2}y + x^{2})h_2^{3}h_3^{36} + xyh_2^{3}h_3^{21} + (xy + x^{2})h_2^{2}h_3^{69} + (y + x)h_2^{2}h_3^{54} + (y + 1)h_2^{2}h_3^{39} + \\
&\hspace{-0.8cm} (y + x)h_2^{2}h_3^{24} + (y + x^{2})h_2^{2}h_3^{9} + (xy + 1)h_2h_3^{87} + h_2h_3^{57} + x^{2}yh_2h_3^{42} + (x^{2}y + x^{2})h_2h_3^{27} + \\
&\hspace{-0.8cm} (x^{2}y + 1)h_2h_3^{12} + h_3^{90} + xh_3^{75} + h_3^{60} + x^{2}h_3^{45} + x^{2}h_3^{30} + 1.\\
\end{split}
\end{equation*}
This equation does not describe the curve $X(1)$, since we did not consider isomorphism classes of $A$-Drinfeld modules yet. Therefore, let $\psi$ be another $A$-Drinfeld module, with the same $A$-characteristic and normalized in the same way as $\phi$, defined by
\begin{align*}
\psi_X &= l_0\tau^4 + l_1\tau^3 + l_2\tau^2 + l_3\tau + \iota(X),\\
\psi_Y &= t_0\tau^4 + t_1\tau^3 + t_2\tau^2 + t_3\tau + \iota(Y).
\end{align*}
An isomorphism between $\phi$ and $\psi$ is a non-zero constant $c$ such that $c\phi = \psi c$. By considering for example the leading coefficient of $c\phi_Y = \psi_Y c$ we get $c^{q^4-1}=1$, implying that
\begin{equation}\label{eq:isominv}
t_1^{(q+1)(q^2+1)} = h_1^{(q+1)(q^2+1)}; t_2^{q^2+1} = h_2^{q^2+1}; t_3^{(q+1)(q^2+1)} = h_3^{(q+1)(q^2+1)}.
\end{equation}
In other words, the quantities $h_1^{(q+1)(q^2+1)}$, $h_2^{q^2+1}$, $h_3^{(q+1)(q^2+1)}$ (and similarly $g_{11}:=g_1^{(q+1)(q^2+1)}$, $g_{22}:=g_2^{q^2+1}$, $g_{33}:=g_3^{(q+1)(q^2+1)}$) are \emph{invariants} of $A$-Drinfeld modules.

Putting $h_{22}:=h_2^{q^2+1}$ and $h_{33}:=h_3^{(q+1)(q^2+1)}$, the previously found equation $f(h_2,h_3)=0$ relating $h_2$ and $h_3$, gives rise to a relation $p(h_{22},h_{33})=0$. One simply uses the relations $f(h_2,h_3),h_{2}^{q^2+1}-h_{22},h_3^{(q+1)(q^2+1)}-h_{33}$ and eliminates the variables $h_2$ and $h_3$ using a Groebner basis computation.  The resulting relation $p(h_{22},h_{33})=0$ then defines the Drinfeld modular curve $X(1)$. This is not immediately clear, since we strictly speaking only can be certain that the function field generated by $h_{22}$ and $h_{33}$ is a subfield of the function field of $X(1)$. However, again using a computer to perform a Groebner basis computation, one can show that this subfield already contains the remaining invariants $h_{11},g_{11},g_{22},$ and $g_{33}$. At first sight it might look s if $\F_{16}(h_{22},h_{33})$ has index $75$ in $\F_{16}(h_{2},h_{3})$. With a computer it can be verified that $h_2$ can be expressed in $h_{22}$ and $h_3$, implying that the index of $\F_{16}(h_{22},h_{33})$ in $\F_{16}(h_{2},h_{3})$ in fact is only $15$, in accordance with the number of possible choices of the isomorphism $c$ mentioned before Equation \eqref{eq:isominv}.

So far, we have computed an explicit model for the curve $X(1)$. The theory implies that this curve has two components. Indeed, according to this prediction, the bivariate polynomial $p(t,s)$ is not absolutely irreducible, but has two absolutely irreducible factors, say $p^1(t,s)$ and $p^2(t,s)$, which turn out to have coefficients in $\F_{16}$. These factors define the curves that we previously denoted by $x^1(1)$ and $x^2(1)$.

To start a recursive description of a tower of function fields, we choose one of the components, say the one defined by $p^1(h_{22},h_{33})=0$ defining the component denoted by $x^1(1)$ . Since this curve has genus zero by Equation \eqref{eq:X1}, its function field is rational and can be described using a parameter $u$, so $\F_{16}(h_{22},h_{33})=\F_{16}(u)$. %It can be expressed explicitly in terms of $h_{22}$ and $h_{33}$ using a computer algebra package like MAGMA.

To describe a tower as in the previous section, we need to choose a prime ideal $\p$. In this section we choose $\p = \langle X,Y \rangle \subset A$, which is coprime with the $A$-characteristic. Since $\deg \p =1$, a $\p$-isogeny $\lambda$ between $\phi$ and $\psi$ is of the form $\tau -a$. From the isogeny property $\lambda\phi_Y = \psi_Y \lambda$ and using as before $x:=\iota(X)$ and $y:=\iota(Y)$, we get
\begin{align}
t_3 &= a^{-q}(y - y^q + ah_3), \\
t_2 &= a^{-q^2} t_3 + a^{1-q^2} h_2 - a^{-q^2} h_3^q.
\end{align}
A direct verification shows that if we set $t_{33}=t_3^{(q+1)(q^2+1)}$ and $t_{22}=t_2^{q^2+1}$ then $t_{33}, t_{22}$ satisfy $p^2(t_{22},t_{33})=0$. In other words, the isogeny maps the component $x^1(1)$ of $X(1)$ to the other component $x^2(1)$. Similar to the uniformizing parameter $u$ of $x^1(1)$, one can find a uniformizing parameter $v$ of $x^2(1)$. Using the above isogeny relation, we can compute $\Phi^1(u,v)=0$ defining $x^1_0(\p)$ like in Figure \ref{fig:uv}.
\begin{figure}[h]
\begin{displaymath}
\xymatrix{
\F_{q^4}(u,h_2,h_3,a) \ar@{-}[d]^{\lambda=\tau-a}\ar@{=}[r]^{\lambda \phi = \psi \lambda} & \F_{q^4}(v,t_2,t_3,a) \ar@{-}[d] \\
\F_{q^4}(u,h_2,h_3) \ar@{-}[d]^{h_{33}=h_3^{(q+1)(q^2+1)} }_{h_{22}=h_2^{q^2+1}} & \F_{q^4}(v,t_2,t_3) \ar@{-}[d]\\
\F_{q^4}(h_{22},h_{33}) = \F_{q^4}(u) & \F_{q^4}(v) = \F_{q^4}(t_{22},t_{33})
}
\end{displaymath}
\caption{Defining $x^1_0(\p)$ explicitly by $\Phi^1(u,v)=0$.}
\label{fig:uv}
\end{figure}
Similarly, starting with the component $x^2(1)$, one finds the relation $\Phi^2(v,w)=0$ defining $x_0^2(\p)$. Explicitly, one obtains:
\begin{align*}
\Phi^1(u,v)&=(u + (x^2 y + 1)) v^3\\
&+ (y u^3 + (x y + 1) u^2 + x^2 y u + (x y + x)) v^2 \\
&+ ((y    + x^2) u^2 + (x^2 y + 1) u + (x y + 1)) v\\
&+ (y + 1) u^3 + x u^2 + y u +  x^2 y + x^2,\\
\Phi^2(v,w)&=(v + x y) w^3\\
&+ ((y + x) v^3 + x^2 y v^2 + x y v + 1) w^2 \\
&+ ((y + 1) v^2 + v + (y + 1)) w\\
&+ (x^2 y + x) v^3 + (y + x) v^2 + (x y + 1) v + x y.
\end{align*}
Now we can construct the tower of function fields $\mathcal{F}=(F_0 \subset F_1 \subset \cdots)$ corresponding to the modular tower $(x_0(\p^k))_k$ by
\begin{enumerate}
\item $F_0 = \F_{16}(u_0)$,
\item $F_1 = F_0(u_1)$ with $\Phi^1(u_0,u_1)=0$.
\item $F_{k} = F_{k-1}(u_k)$ where $\Phi^1(u_{k-1},u_k)=0$ if $k$ odd, $\Phi^2(u_{k-1},u_k)=0$ otherwise.
\end{enumerate}
As remarked in Section \ref{sec:recursive}, for $k>1$, the equations $\Phi^i(u_{k-1},u_k)=0$ give rise to two possible factors: one of degree one in $u_k$ and one of degree $|\p|=q=2$. The factor of degree $2$ should be chosen when defining the tower.

\section{Conclusion}
In this paper we give a recursive description of (reductions of) Drinfeld modular towers for any possible base ring $A$ as well as a lower bound for the limit of such towers. It turns out that good reductions of Drinfeld modular towers are always good, when defined over the proper constant field, but not always optimal. The theory presented here fully explains the behaviour of a Drinfeld modular tower given in \cite{bassa2014a}. Furthermore, an explicit recursive description of an optimal Drinfeld tower over $\F_{16}$ that has not been considered in the literature before is given. This further demonstrates that explicit descriptions of Drinfeld modular towers are not restricted to the case that the base ring $A$ is the polynomial ring.

%\section{Acknowledgement}
%The authors would like to thank the anonymous referee for helpful suggestions and comments, that helped to improve the paper. The last two authors gratefully acknowledge the support from the Danish National Research Foundation and the National Science Foundation of China (Grant No.11061130539) for the Danish-Chinese Center for Applications of Algebraic Geometry in Coding Theory and Cryptography as well as the support from The Danish Council for Independent Research (Grant No. DFF--4002-00367). The first author is supported by Tubitak Proj. No. 112T233.

%\oneappendix % use \appendix if you have more than one appendix
%\section{About the bibliography}
%References in the bibliography should be listed alphabetically by
%the authors' surname(s) and, for the same set of authors, by
%publication year. Detailed formatting (italic, etc.) should be
%avoided; please concentrate on giving full and clear information,
%such as (for books) the \textit{name} and \textit{location} of the
%publisher and (for a book in a book series) the \textit{volume
%number}. Do not include papers `in preparation' in the bibliography;
%these are better mentioned in the main text only.

\section{Acknowledgments}
The authors would like to thank the anonymous referee for helpful suggestions and comments, that helped to improve the paper. The last two authors gratefully acknowledge the support from the Danish National Research Foundation and the National Science Foundation of China (Grant No.11061130539) for the Danish-Chinese Center for Applications of Algebraic Geometry in Coding Theory and Cryptography as well as the support from The Danish Council for Independent Research (Grant No. DFF--4002-00367). The first author is supported by Tubitak Proj. No. 112T233.

\vspace{1ex}

\bibliographystyle{plain}
\bibliography{Drinfeld_modular_curve}

\noindent
Alp Bassa\\
Bo\u{g}azi\c{c}i University,
Faculty of Arts and Sciences,
Department of Mathematics,
34342 Bebek, \.{I}stanbul,
Turkey,
 alp.bassa@boun.edu.tr
\vspace{1ex}

\noindent
Peter Beelen and Nhut Nguyen\\
Technical University of Denmark,
Department of Applied Mathematics and Computer Science,
Matematiktorvet 303B, 2800 Kgs. Lyngby,
Denmark, pabe@dtu.dk, nhngu@dtu.dk
\end{document}